\documentclass[onefignum,onetabnum]{siamart171218}

\usepackage{amsfonts}
\usepackage{amsopn}
\usepackage{booktabs}
\usepackage{graphicx}
\usepackage{hyperref}
\usepackage{enumerate}
\usepackage{cases}

\newsiamremark{remark}{Remark}

\newcommand{\veps}{\varepsilon}
\newcommand{\tr}{\mathrm{tr}}
\newcommand{\ob}{\mathcal{OB}}
\newcommand{\diag}{\mathrm{diag}}

\newcommand{\re}{\mathrm{e}}
\newcommand{\ri}{\mathrm{i}}
\newcommand{\rep}{\mathrm{Re}}  % real part
\newcommand{\imp}{\mathrm{Im}}  % imaginary part
\newcommand{\cg}[1]{#1^*}  % conjugate
\newcommand{\fnorm}[1]{\left\|#1\right\|_\mathrm{F}}  % Frobenius norm

\newcommand{\tA}{\widetilde{A}}
\newcommand{\tB}{\widetilde{B}}
\newcommand{\tLambda}{\widetilde{\Lambda}}
\newcommand{\tlambda}{\tilde{\lambda}}

\renewcommand{\th}{\tilde{h}}
\newcommand{\tSigma}{\widetilde{\Sigma}}
\newcommand{\tX}{\widetilde{X}}
\newcommand{\tV}{\widetilde{V}}
\newcommand{\tH}{\widetilde{H}}

\newcommand{\bbC}{\mathbb{C}}
\newcommand{\bbR}{\mathbb{R}}

\newcommand{\calI}{\mathcal{I}}

\newcommand{\dmin}{d_{\min}}
\newcommand{\dmax}{d_{\max}}
\newcommand{\lmin}{\lambda_{\min}}
\newcommand{\lmax}{\lambda_{\max}}

\title{On the continuity of Schur-Horn mapping}

\headers{Continuity of Schur-Horn}{Hengzhun Chen and Yingzhou Li}

\author{Hengzhun Chen\thanks{School of Mathematical Sciences, Fudan
University (\email{hengzhunchen21@m.fudan.edu.cn}).}
\and Yingzhou Li\thanks{School of Mathematical Sciences, Fudan University;
Shanghai Key Laboratory for Contemporary Applied Mathematics, Fudan
University (\email{yingzhouli@fudan.edu.cn}).} }

\begin{document}

\maketitle

\begin{abstract}
    The Schur-Horn theorem is a well-known result that characterizes the
    relationship between the diagonal elements and eigenvalues of a
    symmetric (Hermitian) matrix. In this paper, we extend this theorem by
    exploring the eigenvalue perturbation of a symmetric (Hermitian)
    matrix with fixed diagonals, which is referred to as the continuity of
    the Schur-Horn mapping. We introduce a concept called strong
    Schur-Horn continuity, characterized by minimal constraints on the
    perturbation. We demonstrate that several categories of matrices
    exhibit strong Schur-Horn continuity. Leveraging this notion, along
    with a majorization constraint on the perturbation, we prove the
    Schur-Horn continuity for general symmetric (Hermitian) matrices. The
    Schur-Horn continuity finds applications in oblique manifold
    optimization related to quantum computing.
\end{abstract}


\section{Introduction}

Schur-Horn theorem was established in the mid-20th century. It continues
to find applications and inspire advancements in various fields of
mathematics and physics. Specifically, the characterization of eigenvalues
and matrix diagonal entries continues to stimulate further research,
driving advancements in quantum information theory, quantum optics,
quantum metrology, spectral graph theory, convex optimization, and
majorization theory. In this paper, we study the continuity of Schur-Horn
mapping, which is adopted for our energy landscape analysis of objective
functions in quantum computing~\cite{chen2025landscape}. 

Schur-Horn theorem is composed of two parts, as proved by Schur and Horn.
We start by defining the condition known as the \emph{majorization}. 

\begin{definition}[Majorization]
    Given vector $x \in \bbR^n$, notation $x^{\uparrow}$ denotes the
    reordered vector of $x$ with entries in non-decreasing order. Let
    $a$ and $\lambda$ be two vectors in $\bbR^n$. Vector $\lambda$ is
    majorized by $a$, denoted as $\lambda \prec a$, if 
    \begin{equation} \label{eq:major}
        \sum_{i=1}^k \lambda^{\uparrow}_{i} \leq \sum_{i=1}^k
        a^{\uparrow}_{i}, \quad k=1, \cdots, n-1; \quad
        \text{and }
        \sum_{i=1}^n \lambda^{\uparrow}_{i} = \sum_{i=1}^n
        a^{\uparrow}_{i}.
    \end{equation}        
\end{definition}
The majorization relation in \eqref{eq:major} is equivalent to another
commonly used definition in the literature, i.e., a vector $a$ is
majorized by $\lambda$, if
\begin{equation*}
    \sum_{i=1}^k a^{\downarrow}_{i} \leq \sum_{i=1}^k
    \lambda^{\downarrow}_{i}, \quad k=1, \cdots, n-1; \quad
    \text{and }
    \sum_{i=1}^n a^{\downarrow}_{i} = \sum_{i=1}^n
    \lambda^{\downarrow}_{i},
\end{equation*}        
where notation $x^{\downarrow}$ denotes the reordered vector of $x$ with
entries in non-increasing order. Then Schur-Horn theorem states:
\begin{itemize}
    \item (Schur~\cite{schur1923uber}) Let $H$ be a Hermitian matrix with
    eigenvalues $\lambda=(\lambda_i)_{1\leq i \leq n}$ and diagonal
    entries $a=(a_{ii})_{1\leq i \leq n}$, then $\lambda \prec a$;

    \item (Horn~\cite{horn1954doubly}) If $a, \lambda \in \bbR^n$ satisfy
    $\lambda \prec a$, then there exists a symmetric (Hermitian) matrix
    $H$ whose diagonal entries are $a$ and eigenvalues are $\lambda$. 
\end{itemize}

Schur-Horn theorem has been applied in various fields. In the realm of
quantum optics and quantum state engineering, the Schur-Horn theorem has
been applied to design and manipulate desired quantum states
\cite{mari2014quantum, styliaris2019mixing}. As for convex optimization,
Schur-Horn theorem has implications in convex relaxations for graph and
inverse eigenvalue problems \cite{candogan2018finding}. The theorem
provides constraints on the eigenvalues of positive semidefinite matrices,
enabling the formulation of optimization problems with eigenvalue
constraints and facilitating the development of efficient algorithms for
solving such problems. 

There has been a rich history in proving the Schur-Horn theorem,
specifically the Horn part. In general, proofs could be grouped into
nonconstructive ones~\cite{chu1995construct, horn1954doubly, LEITE1999149}
and constructive ones~\cite{CHAN1983562, zha1995note}.
Chu~\cite{chu1995construct} utilized an optimization-based limiting
process to prove the existence of the matrix in the Horn part. Leite et
al.~\cite{LEITE1999149} gave an algebraic proof, which could be extended
to analogous results for skew-symmetric matrices as well. Constructive
proofs~\cite{CHAN1983562, zha1995note} were based on Givens rotation and
could be viewed as an algorithm for constructing $H$ in Horn part given
$a$ and $\lambda$ satisfying the majorization condition. Generalizations
of constructive algorithms  
can be found in~\cite{davies2000numerically, dhillon2005generalized}.
Recently, Matthew Fickus et al.~\cite{fickus2013constructing} proposed
an algorithm based on the finite frame theory to procedure every example
of the matrix in the Horn part.

\subsection{Contribution}

Schur-Horn theorem establishes the connections among diagonal entries,
eigenvalues, and a symmetric matrix. In the following, we first define a
Schur-Horn mapping based on the Schur-Horn theorem and then prove the
continuity of the mapping.

Given a target diagonal vector $d \in \bbR^n$, we define two sets of
matrices~\footnote{For the sake of notation, we adopt $\diag(\cdot)$
similar to the MATLAB ``diag'' function, i.e., $\diag(v)$ is a square
diagonal matrix with the entries of vector $v$ on the diagonal and
$\diag(A)$ is a column vector of the diagonal entries of $A$.}
\begin{equation*}
    S_d = \left\{ \Lambda \in \bbR^{n\times n} : \Lambda
    \text{ diagonal }, \diag(\Lambda) \prec d \right\}
\end{equation*}
and 
\begin{equation*}
    M_d = \left\{ A\in \bbR^{n\times n} : \diag(A)=d, A = A^\top \right\}.
\end{equation*}
One can define an equivalence relation between two matrices $A_1$ and
$A_2$ over $M_d$ as 
\begin{equation} \label{eq:equivalence-relation}
    A_1 \sim A_2 \text{ if } A_1, A_2 \text{ have the same
    eigenvalues.}
\end{equation}
Then, we can define a mapping between $S_d$ and the quotient space of
$M_d$ with the equivalence relation~\eqref{eq:equivalence-relation}, whose
existence is guaranteed by the Schur-Horn theorem,
\begin{equation} \label{def:SHmapping}
    \begin{split}
        F : \quad & S_d \rightarrow M_d / \sim \\
        & \Lambda \mapsto [Q \Lambda Q^\top], \quad Q
        \text{ is an orthogonal matrix such that } \diag(Q \Lambda Q^\top) = d.
    \end{split}    
\end{equation}
The mapping $F$ is called the Schur-Horn mapping. Furthermore, we
introduce the Hausdorff distance with the Frobenius norm $\fnorm{\cdot}$,
i.e.,
\begin{equation} \label{def:Hausdorff_dist}
    d_\mathrm{H}([A_1], [A_2]) := \max\left\{ \sup_{X\in [A_1]}
    \inf_{Y\in [A_2]} \fnorm{X - Y}, \sup_{Y\in [A_2]}
    \inf_{X\in [A_1]} \fnorm{X - Y} \right\}.
\end{equation}
This Hausdorff distance measures the distance between two elements in
$M_d/\sim$. We also remark that actually the $\sup$ and $\inf$ in
\eqref{def:Hausdorff_dist} can be replaced by $\max$ and $\min$
respectively since $[A_1]$ and $[A_2]$ are both compact sets in $M_d$ and
$f(Y) = \fnorm{X-Y}$ and $g(X) = \min_{Y \in [A_2]} \fnorm{X-Y}$ are
continuous functions. Indeed, Hausdorff distance is a metric over the set
of compact subsets~\cite{viroelementary}. With the Hausdorff distance being 
a properly defined metric in $M_d/\sim$, we can claim the continuity of
the Schur-Horn mapping $F(\cdot)$, i.e., if $\Lambda_1, \Lambda_2 \in S_d$
are sufficiently close, then $F(\Lambda_1), F(\Lambda_2)$ can be close enough
under the Hausdorff distance. Rigorously, we will first establish a
perturbative analysis for $F(\cdot)$ in \cref{thm:continuitySH} and then
state the continuity of the Schur-Horn mapping in \cref{cor:continuitySH},
which are the main contributions of this paper.  

For ease of expression and reference later, we introduce the definition of
Schur-Horn continuity as follows.

\begin{definition}[Schur-Horn Continuity] \label{def:SHcont}
    Suppose $A$ is an $n$-by-$n$ real symmetric (complex Hermitian) matrix
    with eigenvalues $\lambda \in \bbR^n$ and diagonal entries $d \in
    \bbR^n$. Let $\tlambda$ be a perturbation of $\lambda$ such that
    $\|\lambda - \tlambda\|_2 = O(\veps)$ for $\veps > 0$ sufficiently
    small and $\tlambda \prec d$. Then, there exists a real symmetric
    (complex Hermitian) matrix $\tB$ with eigenvalues $\tlambda$ and
    diagonal entries $d$ such that $\fnorm{A - \tB} = O(\veps^{1/2})$.
\end{definition}

\begin{theorem} \label{thm:continuitySH}
    Any symmetric matrix $A \in \bbR^{n \times n}$ is Schur-Horn
    continuous.
\end{theorem}

From \cref{thm:continuitySH}, one could easily deduce the continuity of
the Schur-Horn mapping as the following corollary.

\begin{corollary} \label{cor:continuitySH}
    Schur-Horn mapping $F$ is a continuous mapping from $S_d$ to $M_d/\sim$
    with Hausdorff distance $d_H$.    
\end{corollary}

\begin{proof}
    Given $\Lambda_1, \Lambda_2 \in S_d$ such that $\fnorm{\Lambda_1 -
    \Lambda_2} = O(\veps)$, denote $[A_1] = F(\Lambda_1)$, $[A_2] =
    F(\Lambda_2)$. According to \cref{thm:continuitySH} we have
    $\min_{Y\in[A_2]}\fnorm{X-Y}=O(\veps^{1/2})$. Note that $g(X) =
    \min_{Y \in [A_2]} \fnorm{X-Y}$ is a continuous function and $[A_1]$
    is a compact set, it yields that 
    \begin{equation*}
        \max_{X\in[A_1]} \min_{Y\in [A_2]} \fnorm{X - Y} = O(\veps^{1/2}).
    \end{equation*}
    Similarly, we have 
    \begin{equation*}
        \max_{Y\in [A_2]} \min_{X\in [A_1]} \fnorm{X - Y}=O(\veps^{1/2}). 
    \end{equation*} 
    Thus, from the definition of Hausdorff distance
    \eqref{def:Hausdorff_dist} we obtain the continuity of Schur-Horn
    mapping.   
\end{proof}

\cref{cor:continuitySH} establishes continuity of the Schur--Horn
mapping $F$ in a set-valued sense on the quotient space $M_d / \sim$,
measured by the Hausdorff distance. While this is a natural and
sufficient notion of continuity for the mapping itself, many
perturbative arguments in constrained optimization require more detailed
information: given a representative matrix $A$ and a small perturbation
of its eigenvalues, one needs to construct an explicit nearby
representative $\tB$ with the same diagonal satisfying other constraints
in the constrained optimization problem, rather than merely asserting
that the two equivalence classes are close. This is particularly
relevant for local saddle-point perturbations on manifolds (see, e.g.,
\cref{app:qomm_application}). To capture this constructive, local
refinement, we introduce in \cref{sec:strongSHcont} the notion of
\emph{strong Schur--Horn continuity}.

For some applications where matrices are Hermitian, we can still define
the Schur-Horn mapping and prove its continuity. Consider 
\begin{equation*}
    N_d = \left\{ A\in \bbC^{n\times n} : \diag(A)=d, A = \cg{A} \right\},
\end{equation*}
we can define Schur-Horn mapping for Hermitian scenario as 
\begin{equation} \label{def:SHmapping_hermit}
    \begin{split}
        F : \quad & S_d \rightarrow N_d / \sim \\
        & \Lambda \mapsto [Q \Lambda \cg{Q}], \quad Q
        \text{ is a unitary matrix such that } \diag(Q \Lambda \cg{Q}) = d.
    \end{split}    
\end{equation}
Here, we abuse the notation $F$ to denote the Schur-Horn mapping for
Hermitian matrices. Then we have similar results as follows.
\begin{theorem} \label{thm:continuitySH_hermit}
    Any Hermitian matrix $A\in \bbC^{n\times n}$ is Schur-Horn continuous.
\end{theorem}

\begin{corollary} \label{cor:continuitySH_hermit}
    Schur-Horn mapping $F$ is a continuous mapping from $S_d$ to $N_d/\sim$
    with Hausdorff distance $d_H$.    
\end{corollary}

\subsection{Applications}

The continuity property of the Schur-Horn mapping is useful in analyzing
the manifold optimization problems. For example, consider the landscape
analysis of an objective function over the oblique manifold. The descent
direction of such a problem has to incorporate the manifold information,
and the perturbative analysis of a stationary point on the manifold
directly links to the continuity of the Schur-Horn mapping. We provide a
concrete application of the continuity of the Schur-Horn mapping.

Given a negative definite Hermitian matrix $A \in \bbC^{n \times n}$, we
consider the following manifold optimization problem,
\begin{equation} \label{eq:qomm_energy}
    \min_{X \in \ob(n,p)}E_0(X) = \tr\left( (2I - X^* X ) X^* A X \right),
\end{equation}
where the oblique manifold is defined as,
\begin{equation} \label{def:obli_mani}
    \ob(n,p) = \left\{ X \in \bbC^{n\times p} \, \big|
    \, \text{diag}(X^* X) = \mathbf{1} \right\},
\end{equation}
where $\mathbf{1}$ denotes an all-one vector of length $p$. The
minimization problem~\eqref{eq:qomm_energy} without the oblique manifold
constraint has been known as the unconstrained orbital minimization
method~(OMM)~\cite{mauri1993orbital, ordejon1993unconstrained} in the
literature, which is used to seek the low-lying eigenpairs of $A$.
In~\cite{bierman2022quantum}, the OMM objective function is adopted in a
variational quantum eigensolver (VQE) on quantum computers, known as
quantum orbital minimization method~(qOMM). The manifold optimization
problem~\eqref{eq:qomm_energy} is the optimization problem of qOMM, where
the oblique manifold constraint appears due to the unitary quantum state
constraint from the quantum computer.

Without the oblique manifold constraint, OMM has an attractive property:
all minima are formed by the eigenvectors of $A$ corresponding to the
low-lying eigenvalues, and it has no spurious local
minima~\cite{LU201787}. With the oblique manifold constraint, we would
like to have the same property. In the study of the first-order stationary
points of \eqref{eq:qomm_energy}, we would like to show that some of them
are strict saddle points, and a local perturbation leads to decay in the
objective function. However, the oblique manifold constraint requires that
the perturbed points have to stay in the manifold, i.e., $\diag(X^*X) =
\mathbf{1}$. The continuity of the Schur-Horn mapping assures that for any
local perturbation on eigenvalues of $X^*X$ there is a corresponding point
in the neighborhood of $X^*X$ on the manifold. Then, the objective
function at a saddle point decays for a particular perturbation, and
hence, the same property of OMM holds for qOMM. For more details of
the application in qOMM, see \cref{app:qomm_application}.

Recently, other objective functions~\cite{gao2022triofm, gao2023wtpm,
li2019coord, liu2015efficient, wen2016trace} have been applied in VQE on
quantum computers and lead to the following constraint optimization
problems,
\begin{equation*}
    \min_{X\in \ob(n,p)} \frac{1}{2}\tr(X^*AX) + \frac{\mu}{4} \fnorm{X^*X- I}^2
    \quad \text{ and } \min_{X\in \ob(n,p)} \frac{1}{2} \fnorm{XX^* - A}^2.
\end{equation*}
In their landscape analysis, the continuity of the Schur-Horn mapping can
be applied to similar scenarios of saddle point analysis. Similar property
as that of OMM could be proved.

\subsection{Organization} 

The rest of the paper is organized as follows. In \cref{sec:SHC_diag}, we
first establish the Schur-Horn continuity for real diagonal matrices.
Then, in \cref{sec:strongSHcont}, we introduce the concept of strong
Schur-Horn continuity and demonstrate its application to specific matrix
types, which plays a crucial role in proving our main theorem. In
\cref{sec:SHC_symmetric}, we combine the previous results to prove the
Schur-Horn continuity for symmetric matrices. The Schur-Horn continuity
for Hermitian matrices is covered in \cref{sec:SHC_Hermitian}. Finally, we
conclude the paper in \cref{sec:conclusion}.

\section{Schur-Horn Continuity of Diagonal Matrices} \label{sec:SHC_diag}

We first prove the Schur-Horn continuity of diagonal matrices, as in
\cref{thm:diagonal_case}. The proof explicitly shows that the majorization
condition plays an essential role in the Schur-Horn continuity.

Before proving \cref{thm:diagonal_case}, we first prove \cref{lem:givens},
which is the generalized version of Schur-Horn continuity for matrices of
size $2 \times 2$. \Cref{lem:givens} will be repeatedly used throughout
this paper. In the following, standard big $O$ and big $\Theta$ notations
are used with $\veps$ being the asymptotic variable approaching zero.
Other variables, including matrix dimensions and matrix nonzero entries,
are viewed as constants.

\begin{lemma} \label{lem:givens}
    Given $\veps >0$ sufficiently small and $d_1, d_2 \in \bbR$. Let $B$
    be a symmetric matrix of form,
    \begin{equation*}
        B = \begin{bmatrix}
            b_{11} & b_{12} \\
            b_{21} & b_{22}
        \end{bmatrix}
        =
        \begin{bmatrix}
            d_1 - f(\veps) & b_{12} \\
            b_{12} & d_2 + g(\veps)
        \end{bmatrix},
    \end{equation*}
    with $f(\veps) = \Theta(\veps^\alpha), g(\veps) =
    \Theta(\veps^{\beta})$ for $\alpha, \beta > 0$. Further, we assume
    that 
    \begin{equation*}
        b_{12}^2 + f(\veps)(d_2 - d_1 + g(\veps)) \geq 0. 
    \end{equation*}
    Then there exists a Givens rotation $G$ with rotation angle $\theta =
    \Theta(\veps^{\gamma})$ such that the $(1, 1)$ entry of $\tB =
    GBG^\top$ is $d_1$ and $\fnorm{\tB - B} = O(\veps^{\delta})$, where
    various scenarios of $\gamma$ and $\delta$ are provided in
    \cref{tab:givens}.
\end{lemma}

\begin{table}[htbp]
    \centering
    \begin{tabular}{lllcc}
        \toprule
        \multicolumn{3}{c}{Various Scenarios} & $\gamma$ & $\delta$ \\
        \toprule
        $b_{12} \neq 0$ & & & $\alpha$ & $\alpha$ \\
        \midrule
        $b_{12} = 0$ & $d_1 \neq d_2$ & & $\alpha / 2$ & $\alpha / 2$ \\
        \midrule
        $b_{12} = 0$ & $d_1 = d_2$ & $\alpha > \beta$ & $(\alpha-\beta)/2$ & $(\alpha + \beta)/2$ \\
        \midrule
        $b_{12} = 0$ & $d_1 = d_2$ & $\alpha \leq \beta$ & $0$ & $\alpha$ \\
        \bottomrule
    \end{tabular}
    \caption{Various scenarios of $b_{12}$, $d_1$, $d_2$, $\gamma$, and
    $\delta$ for \cref{lem:givens}.}
    \label{tab:givens}
\end{table}

\begin{proof}
    Denote the Givens rotation matrix as $G = \begin{bmatrix} c & s \\ 
    -s & c \end{bmatrix}$, where $c = \cos \theta$ and $s = \sin \theta$.
    Then we have 
    \begin{equation} \label{eq:tildeB}
        \tB = GBG^\top = \begin{bmatrix}
            c^2b_{11} + s^2b_{22} + 2csb_{12} &
            (c^2-s^2) b_{12} + cs(b_{22}-b_{11}) \\
            (c^2-s^2)b_{12} + cs(b_{22}-b_{11}) &
            c^2b_{22} + s^2b_{11} - 2csb_{12}
        \end{bmatrix}.
    \end{equation}
    Equating the $(1,1)$ entry of $\tB$ and $d_1$, we obtain,
    \begin{equation} \label{eq:secondorderoriginal}
        c^2 b_{11} + s^2 b_{22} + 2csb_{12} = d_1.
    \end{equation}
    Denoting $t = \frac{s}{c} = \tan \theta$, we reformulate
    \eqref{eq:secondorderoriginal} into
    \begin{equation} \label{eq:secondorder}
        (b_{22} - d_1) t^2 + 2 b_{12} t + b_{11} - d_1 = 0.
    \end{equation}
    Equation \eqref{eq:secondorder} has real solutions if and only if
    \begin{equation} \label{eq:delta}
        \Delta = 4(b_{12}^2 - (b_{22} - d_1)(b_{11} - d_1))
        = 4(b_{12}^2 + f(\veps)(d_2 - d_1 + g(\veps))) \geq 0.
    \end{equation}
    According to the assumption in the lemma, we know that
    \eqref{eq:secondorder} always has at least one solution, thus $G$
    exists. Furthermore, by solving the quadratic equation, one has
    \begin{equation} \label{eq:quadratic_solution}
        t = \frac{f(\veps)}{b_{12} \pm \sqrt{b_{12}^2 + f(\veps)(d_2 - d_1
        + g(\veps))}}.
    \end{equation}

    Now, we divide the discussion into various scenarios, as in
    \cref{tab:givens}.
    
    \emph{Case 1: $b_{12} \neq 0$.}

    In this case, $\Delta$ is positive when $\veps$ is sufficiently small.
    The quadratic equation~\eqref{eq:secondorder} then has two solutions in
    \eqref{eq:quadratic_solution}, and one of them is of order
    $\frac{f(\veps)}{2b_{12}}$. Hence, one solution gives $\tan \theta =
    \Theta(\veps^\alpha)$, thus $\theta = \Theta(\veps^\alpha)$. By the
    triangular inequality of matrix norm, we obtain,
    \begin{align*}
        \fnorm{\tB - B} & \leq \fnorm{GBG^\top - GB} + \fnorm{GB - B} \\
        & \leq \fnorm{GB} \cdot \fnorm{G^\top - I} + 
        \fnorm{G-I} \cdot \fnorm{B} = O(\veps^{\alpha}).
    \end{align*}

    \emph{Case 2: $b_{12} = 0$ and $d_2 \neq d_1$.}

    When $b_{12} = 0$, we have  
    \begin{equation*}
        \tB - B = 
        \begin{bmatrix}
            f(\veps) & 
            cs(d_2 - d_1 + g(\veps) + f(\veps)) \\
            cs(d_2 - d_1 + g(\veps) + f(\veps)) &
            - f(\veps)
        \end{bmatrix}.
    \end{equation*}
    At this time, the solutions \eqref{eq:quadratic_solution} of the quadratic
    equation \eqref{eq:secondorder} admit,
    \begin{equation*}
        t = \pm \sqrt{\frac{f(\veps)}{d_2 - d_1 + g(\veps)}}
        = \Theta(\veps^{\alpha/2}),
    \end{equation*}
    where the positivity of the quantity under the square root is guaranteed by
    \eqref{eq:delta}. Thus, we can derive $\theta = \Theta(\veps^{\alpha/2})$
    and $\fnorm{\tB - B} = O(\veps^{\alpha/2})$.

    \emph{Case 3 and Case 4: $b_{12} = 0$ and $d_2 = d_1$.}

    Denoting $d = d_1 = d_2$, the solutions
    \eqref{eq:quadratic_solution} of the quadratic equation
    \eqref{eq:secondorder} admit,
    \begin{equation*}
        t = \pm \sqrt{\frac{f(\veps)}{g(\veps)}} 
        = \Theta(\veps^{(\alpha - \beta)/2}),
    \end{equation*}    
    where the positivity of the quantity under the square root is also
    guaranteed by \eqref{eq:delta}. Additionally, we have 
    \begin{equation} \label{eq:entries_diff_case34}
        \tB - B =
        \begin{bmatrix}
            f(\veps) & cs(g(\veps) + f(\veps)) \\
            cs(g(\veps) + f(\veps)) & -f(\veps)
        \end{bmatrix},
    \end{equation}
    where $|s| \leq 1$ and $|c| \leq 1$. 
    
    When $\alpha > \beta$ and $\veps$ sufficiently small, we have $\theta
    = \Theta(\veps^{(\alpha-\beta)/2})$. However, consider entries
    comparison of $\tB$ and $B$ as \eqref{eq:entries_diff_case34}, one can
    deduce that 
    \begin{equation*}
        cs(g(\veps)+f(\veps)) =
        \Theta(\veps^{(\alpha-\beta)/2+\beta})=\Theta(\veps^{(\alpha+\beta)/2}),
    \end{equation*}
    with $f(\veps)=\Theta(\veps^{\alpha})$. Note that $\alpha > \beta$ implies
    $\frac{\alpha + \beta}{2} < \alpha$, thus we have $\fnorm{\tB - B} =
    O(\veps^{(\alpha + \beta)/2})$.
    
    When $\alpha \leq \beta$, we have $\theta = \Theta(1)$. Comparing entries of
    $\tB$ and $B$ as \eqref{eq:entries_diff_case34}, we could see that the norm
    of $\tB - B$ is bounded by the lower order of $f(\veps)$ and $g(\veps)$, and
    we have $\fnorm{\tB - B} = O(\veps^{\alpha})$.
\end{proof}

\begin{remark}
    \Cref{lem:givens} could be directly extended to the case where one
    (or both) of $\alpha,\beta$ is $\infty$, i.e., when $f(\veps)$ or
    $g(\veps)$ decays faster than any polynomial in $\veps$. Throughout
    the paper, when invoking \cref{lem:givens} we use the convention
    that a generic perturbation size $O(\veps)$ can be treated as
    $\Theta(\veps^{\alpha})$ for some $\alpha\in[1,\infty]$. Under this
    convention, the bounds $\fnorm{\tB-B}=O(\veps^{\delta})$ listed in
    \cref{tab:givens} remain valid for all $\alpha,\beta\in[1,\infty]$
    (including $\infty$). For brevity, we will not distinguish between
    polynomial and super-polynomial decay in the subsequent estimates.
\end{remark}

\begin{theorem} \label{thm:diagonal_case}
    Any diagonal matrix $A\in \bbR^{n\times n}$ is Schur-Horn continuous.
\end{theorem}

\begin{proof}
    Without loss of generality, we assume that the diagonal entries of $A$
    are non-decreasing, i.e., $d_1 \leq d_2 \leq \cdots \leq d_n$. The
    perturbed matrix is denoted as $\tA^{(0)}$. The perturbed eigenvalues
    are denoted as $\tlambda_i = d_i + h_i(\veps)$. When the eigenvalues
    have a gap, i.e., $\lambda_i < \lambda_{i+1}$, the perturbed
    eigenvalues keep the ordering, i.e., $\tlambda_i \leq \tlambda_{i+1}$
    for sufficiently small $\veps$. When eigenvalues are identical,
    $\lambda_i = \lambda_{i+1}$, the perturbed eigenvalues are ordered
    based on their perturbations. Then we have the majorization relations
    given $\veps > 0$ sufficiently small 
    \begin{align*}
        \tlambda_1 & \leq d_1, \\
        \tlambda_1 + \tlambda_2 & \leq d_1 + d_2, \\
        &\,\,\,\vdots \\
        \tlambda_1 + \cdots + \tlambda_{n-1} & \leq d_1 + \cdots + d_{n-1}, \\
        \tlambda_1 + \cdots + \tlambda_{n-1} + \tlambda_n
        & = d_1 + \cdots + d_{n-1} + d_n.
    \end{align*}
    Hence we have
    \begin{eqnarray*}
        & & h_1(\veps) + h_2(\veps) + \cdots + h_i(\veps) \leq 0, \quad
        i=1, \cdots, n-1, \text{ and} \\ & & h_1(\veps) + \cdots +
        h_n(\veps) = 0.
    \end{eqnarray*}
    Next, we describe a procedure to correct the diagonal entries from
    perturbed $\tlambda_i$ to $d_i$.
    
    We maintain a priority queue with diagonal indices as elements. For
    any diagonal index $i$ in the queue, we ensure that $d_i +
    \th_i(\veps)$ has a negative perturbation $\th_i(\veps) < 0$, where
    $\th_i(\veps)$ denotes the updated perturbation throughout the
    procedure. Starting from the first diagonal entry, we check and
    enqueue the index $i = 1, 2, \dots$ in order if $\th_i(\veps) < 0$ and
    skip the index $i$ if $\th_i(\veps) = 0$. We keep on checking and
    enqueuing indices until the first index $j$ such that $\th_j(\veps) >
    0$. If $j$ does not exist, then by the last equation in majorization
    relation, we know that the queue is also empty and the diagonal
    entries of the perturbed $A$ have all been corrected. Otherwise, we
    obtain a $j$ and the updated perturbations satisfy,
    \begin{equation} \label{eq:updated-majorization}
        \th_1(\veps) + \cdots + \th_j(\veps) = h_1(\veps) + \cdots + 
        h_j(\veps) \leq 0.
    \end{equation}
    This condition is satisfied in the first correction step and we will
    verify it after each step. By the $j$-th majorization relation, the
    queue is guaranteed to be non-empty. We pop an index from the queue
    and denote it as $i$. The current working matrix is denoted as
    $\tA^{(i-1)}$.\footnote{If the index $i$ is skipped, we assign
    $\tA^{(i)} = \tA^{(i-1)}$.}
    
    Based on the property of $i$ and $j$, the inequality assumption in
    \cref{lem:givens} is always satisfied. Hence, the lemma provides a
    Givens rotation applying to the $i$-th and $j$-th columns and rows to
    correct the $(i, i)$ diagonal entry. Without loss of generality,
    applying the extended Givens rotation matrix symmetrically to the
    current matrix $\tA^{(i-1)}$ we obtain $\tA^{(i)}$, whose top-left
    $j$-by-$j$ submatrix admits,
    \begin{equation} \label{eq:rotate_diag_case}
        \begin{split}
            &
            \begin{bmatrix}
                I & & & \\ 
                & c & & s \\
                & & I & \\
                & -s & & c \\
            \end{bmatrix}
            \renewcommand\arraystretch{1.5}
            \begin{bmatrix}
                \tA_{\calI_1 \calI_1} & p_{i \calI_1}^\top &
                \tA_{\calI_2 \calI_1}^\top & p_{j \calI_1}^\top \\
                p_{i \calI_1} & d_i + \th_i(\veps) & p_{\calI_2 i}^\top & 0 \\
                \tA_{\calI_2 \calI_1} & p_{\calI_2 i} & \tA_{\calI_2 \calI_2} &
                p_{j \calI_2}^\top \\
                p_{j \calI_1} & 0 & p_{j \calI_2} & d_j + \th_j(\veps) \\
            \end{bmatrix}
            \renewcommand\arraystretch{1}
            \begin{bmatrix}
                I & & & \\ 
                & c & & -s \\
                & & I & \\
                & s & & c \\
            \end{bmatrix} \\
            = &
            \renewcommand\arraystretch{1.5}
            \begin{bmatrix}
                \tA_{\calI_1 \calI_1} &
                cp_{i \calI_1}^\top + sp_{j \calI_1}^\top &
                \tA_{\calI_2 \calI_1}^\top &
                cp_{j \calI_1}^\top - sp_{i \calI_1}^\top \\
                cp_{i \calI_1} + sp_{j \calI_1} & d_i &
                cp_{\calI_2 i}^\top + sp_{j \calI_2} & p_{ij} \\
                \tA_{\calI_2 \calI_1} & cp_{\calI_2 i} + sp_{j \calI_2}^\top &
                \tA_{\calI_2 \calI_2} & cp_{j \calI_2}^\top - sp_{\calI_2 i} \\
                cp_{j \calI_1} - sp_{i \calI_1} & p_{ji} &
                cp_{j \calI_2} - sp_{\calI_2 i}^\top &
                d_j + \th_i(\veps) + \th_j(\veps) \\
            \end{bmatrix},
        \end{split}
    \end{equation}
    where $\calI_1 = \{1, \ldots, i-1\}$, $\calI_2 = \{i+1, \ldots,
    j-1\}$, vectors $p_{i \calI_1}$, $p_{\calI_2 i}$, $p_{j \calI_1}$ and
    $p_{j \calI_2}$ are all perturbations introduced in previous
    steps, submatrices $\tA_{\calI_1 \calI_1}$, $\tA_{\calI_2 \calI_1}$,
    and $\tA_{\calI_2 \calI_2}$ are untouched submatrices of
    $\tA^{(i-1)}$,\footnote{We dropped the superscript $(i-1)$ for
    simplicity.} and
    \begin{equation*}
        p_{ij} = p_{ji} = cs(d_j - d_i + \th_j(\veps) - \th_i(\veps)).
    \end{equation*}
    The Givens rotation above corrects the $(i,i)$ diagonal entry and it
    only changes the top-left $j$-by-$j$ submatrix of $\tA^{(i-1)}$,
    leaving the remain part of $\tA^{(i-1)}$ unchanged. Since $\|\lambda -
    \tlambda\|_2=O(\veps)$, from the correction procedure we have 
    \begin{equation} \label{eq:order_thi_thj}
        \th_i(\veps)=O(\veps)=\Theta(\veps^{\alpha}), \quad 
        \th_j(\veps)=O(\veps)=\Theta(\veps^{\beta}),    
    \end{equation}
    with $\alpha \geq 1$ and $\beta \geq 1$. 

    We now have three cases: i) $\th_i(\veps) + \th_j(\veps) < 0$; ii)
    $\th_i(\veps) + \th_j(\veps) = 0$; and iii) $\th_i(\veps) +
    \th_j(\veps) > 0$. In case i), we enqueue $j$ and start checking the
    following indices. In case ii), we skip $j$ and start checking the
    indices after $j$. In case iii), we pop another index from the queue
    and repeat the correction procedure. In all cases, the updated
    perturbation at index $i$ and $j$ are $0$ and $\th_i(\veps) +
    \th_j(\veps)$, respectively. Hence \eqref{eq:updated-majorization}
    holds for all indices greater or equal to $j$. Then, the majorization
    relations guarantee that the procedure ends if and only if all
    diagonal entries have been corrected.

    Finally, we show that the corrected matrix is within an $\veps^{1/2}$
    neighborhood of the original matrix. In the above procedure, each step
    corrects at least one diagonal index, and the procedure finishes in at
    most $n$ steps for $n$ being the matrix size. Now, we show that all
    the off-diagonals of $\tA^{(i-1)}$ are $O(\veps^{1/2})$ for $i=1,
    \ldots, n$, by induction. It is obvious that all the off-diagonals of
    $\tA^{(0)}=\diag(\tlambda)$ are zeros and hence $O(\veps^{1/2})$,
    which gives the start point of induction. For those $\th_i(\veps)=0$,
    $\tA^{(i)}=\tA^{(i-1)}$, hence we only need to consider the case when
    $\th_i(\veps)\neq 0$. Note that vectors $p_{i \calI_1}$, $p_{\calI_2
    i}$, $p_{j \calI_1}$ and $p_{j \calI_2}$ in
    \eqref{eq:rotate_diag_case} are $O(\veps^{1/2})$ by our induction
    assumption, it implies that the off-diagonals of $\tA^{(i)}$ are
    $O(\veps^{1/2})$ except $p_{ij}$ and $p_{ji}$. Denote
    \begin{equation*}
        B = \begin{bmatrix}
            d_i + \th_i(\veps) & 0 \\ 0 & d_j + \th_j(\veps)
        \end{bmatrix}, \quad
        \tB = \begin{bmatrix}
            d_i & p_{ij} \\ p_{ji} & d_j + \th_j(\veps) + \th_i(\veps)
        \end{bmatrix},
    \end{equation*}
    we split the discussion according to the scenarios with $b_{12}=0$ in
    \cref{lem:givens} and \eqref{eq:order_thi_thj} as follow:
    \begin{enumerate}[(i)]
        \item $d_i \neq d_j$, $\fnorm{B-\tB} = O(\veps^{\alpha/2}) =
        O(\veps^{1/2})$; 

        \item $d_i = d_j$ and $\alpha > \beta$, $\fnorm{B - \tB} =
        O(\veps^{(\alpha + \beta) / 2})= O(\veps^{1/2})$;

        \item $d_i = d_j$ and $\alpha \leq \beta$, $\fnorm{B - \tB} =
        O(\veps^{\alpha})= O(\veps^{1/2})$.
    \end{enumerate}
    Therefore, we conclude that all the off-diagonals of $\tA^{(i)}$ are
    $O(\veps^{1/2})$. Furthermore, note that the diagonals of $\tA^{(i)}$
    and $\tA^{(i-1)}$ only differ at indices $i$ and $j$, together with
    $\fnorm{B - \tB}=O(\veps^{1/2})$, it implies that $\fnorm{\tA^{(i)} -
    \tA^{(i-1)}}=O(\veps^{1/2})$. By triangular inequality of Frobenius
    norm, the distance between the corrected matrix and the original
    matrix is bounded as,
    \begin{equation*}
        \fnorm{A - \tA^{(n)}} \leq
        \fnorm{A - \tA^{(0)}} + \sum_{k = 1}^n \fnorm{\tA^{(k-1)} - \tA^{(k)}}
        = O(\veps^{1/2}).
    \end{equation*}
    Thus, $A$ is Schur-Horn continuous.
\end{proof}

\begin{remark}
    \Cref{thm:diagonal_case} is covered by our main
    \cref{thm:continuitySH_hermit}. And the proof of
    \cref{thm:diagonal_case} has the same structure as that of
    \cref{thm:continuitySH_hermit}, where more complicated scenarios are
    discussed in the later one. We present the diagonal matrix case as a
    stand-alone theorem to facilitate the understanding of the main
    theorem proof. The diagonal matrix perturbation is also used to prove
    the strong Schur-Horn continuity.
\end{remark}

\section{Strong Schur-Horn Continuity} \label{sec:strongSHcont}

In this section, we define the strong Schur-Horn continuity. The strong
Schur-Horn continuity is a stronger version of the Schur-Horn
continuity, which plays a key role in the proof of Schur-Horn continuity
of general symmetric matrices. And we will prove that if a matrix is
strongly Schur-Horn continuous then it is Schur-Horn continuous, but not
the other way around. Before delving into the proofs, it is essential to
define the spectrum window and strong Schur-Horn continuity.

\begin{definition} [Spectrum Window]
    Let $A$ be a symmetric matrix. The spectrum window of $A$ is defined
    as the closed interval of the minimum and maximum eigenvalues of $A$,
    i.e., \footnote{We use $\lmin(\cdot)$ and $\lmax(\cdot)$ to denote the
    minimum and maximum eigenvalues of a matrix, respectively.}
    \begin{equation*}
        \omega(A) := [\lmin(A), \lmax(A)].
    \end{equation*}
\end{definition}

\begin{definition}[Strong Schur-Horn Continuity] \label{def:strongSHcont}
    Suppose $A \in \bbR^{n \times n}$ is a symmetric matrix with an
    eigendecomposition $A = Q \Lambda Q^\top$, where $Q$ is the
    orthonormal eigenvector matrix and $\Lambda$ is the diagonal
    eigenvalue matrix. Matrix $A$ is strongly Schur-Horn continuous if,
    for any perturbed eigenvalues $\tLambda$ satisfying $\tr(\tLambda) =
    \tr(\Lambda)$ and $\fnorm{\tLambda - \Lambda} = O(\veps)$ for $\veps >
    0$ sufficiently small, there exists a symmetric matrix $\tB = G_2 Q
    G_1 \tLambda G_1^\top Q^\top G_2^\top$ such that
    \begin{enumerate}
        \item $\diag(\tB)=\diag(A)$,
        \item $G_1$ and $G_2$ are orthogonal matrices, and
        \item $\fnorm{G_i - I} = O(\veps^{1/2})$ for $i= 1, 2$.
    \end{enumerate}
\end{definition}

With relaxed requirements, it is straightforward to conclude that the
strong Schur-Horn continuity directly implies the Schur-Horn continuity.

\begin{corollary} \label{cor:strongSH-SH}
    If a matrix is strongly Schur-Horn continuous, then it is Schur-Horn
    continuous.
\end{corollary}

\begin{remark}
    There exists a counterexample matrix $A$ being Schur-Horn continuous
    but not strongly Schur-Horn continuous. Consider the matrix $A$ and
    its perturbed eigenvalue matrix $\tLambda$,
    \begin{equation*}
        A =
        \begin{bmatrix}
            1 & 0 \\
            0 & 2
        \end{bmatrix} \text{ and }
        \tLambda =
        \begin{bmatrix}
            1+\veps & 0 \\
            0 & 2-\veps
        \end{bmatrix},
    \end{equation*}
    where $\veps > 0$ is sufficiently small, and the perturbation
    satisfies the last majorization relation. By \cref{thm:diagonal_case},
    the diagonal matrix $A$ is Schur-Horn continuous. However, for the
    given perturbation above, the first majorization relation is violated.
    By Schur theorem, there does not exist a matrix $\tB$ whose eigenvalue
    matrix is $\tLambda$ and diagonal entries being 1 and 2. Hence, we
    conclude that $A$ is not strongly Schur-Horn continuous.
\end{remark}

Comparing \cref{def:strongSHcont} and \cref{def:SHcont}, there are two
differences. First, the requirement of the perturbation is relaxed in
\cref{def:strongSHcont}. In \cref{def:SHcont}, the perturbation needs to
satisfy all majorization relations as in \eqref{eq:major}. While,
\cref{def:strongSHcont} only requires the perturbation to satisfy the last
majorization relation. This difference is essential between the Schur-Horn
continuity and the strong Schur-Horn continuity. Furthermore, we have
\cref{prop:strongSH-major} describing the relation of eigenvalues and
diagonal entries of a strongly Schur-Horn continuous matrix, which
directly extends the result of the Schur part in the Schur-Horn theorem.
The proof of \cref{prop:strongSH-major} can be found in
\cref{app:strongSH-major-proof}.

\begin{proposition} \label{prop:strongSH-major}
    Suppose matrix $A\in \bbR^{n\times n}$ is strongly Schur-Horn
    continuous with eigenvalues $\lambda\in \bbR^n$ and diagonal entries
    $d\in\bbR^n$. Without loss of generality, both $\lambda$ and $d$ are
    in non-decreasing order, then either $A$ is a scalar matrix, i.e., 
    \begin{equation*}
        \lambda_1 = \cdots = \lambda_n = d_1 = \cdots = d_n;
    \end{equation*}
    or $A$ satisfies the strict majorization relation, i.e., the first
    $n-1$ majorization inequalities of $A$ are strict, 
    \begin{equation}
        \label{eq:strict-major}
        \begin{split}
        \lambda_1 & < d_1, \\
        \lambda_1 + \lambda_2 & < d_1 + d_2, \\
        &\,\,\,\vdots \\
        \lambda_1 + \cdots + \lambda_{n-1} & < d_1 + \cdots + d_{n-1}, \\
        \lambda_1 + \cdots + \lambda_{n-1} + \lambda_n
        & = d_1 + \cdots + d_{n-1} + d_n.
        \end{split}
    \end{equation}
\end{proposition}

The second difference is the transformation of $\tLambda$, i.e., the
eigenvector matrix of $\tB$. In \cref{def:SHcont}, no restriction is
applied to the eigenvector matrix of $\tB$, whereas in
\cref{def:strongSHcont}, the eigenvector matrix is required to be an
$\veps$ perturbation of the original eigenvector matrix of $A$. The second
difference is not essential. The explicit expression $G_2 Q G_1$
simplifies our later proofs.

Next, we provide a few lemmas that prove specific types of matrices $A$
are strongly Schur-Horn continuous. These lemmas contribute to the final
proof of our main theorem.

\begin{lemma} \label{lem:irreducible_case}
    Any irreducible symmetric matrix $A\in \bbR^{n\times n}$ is strongly
    Schur-Horn continuous.
\end{lemma}

\begin{proof}
    Denote the eigendecomposition of $A$ as $A = Q \Lambda Q^\top$, where
    $Q$ is the eigenvectors of $A$ and $\Lambda$ is the diagonal
    eigenvalue matrix. The diagonal entries of $A$ is denoted as $d \in
    \bbR^n$. From $\fnorm{\tLambda-\Lambda}=O(\veps)$, we construct the
    perturbed matrix $B = Q \tLambda Q^\top$ with $\fnorm{B - A} =
    O(\veps)$. When $\veps$ is sufficiently small, we know that the
    nonzero entries in $A$ remain nonzero in $B$ and are $\Theta(1)$ with
    respect to $\veps$. Our proof starts from $B$ and constructs the
    desired $\tB$ step by step based on \cref{lem:givens}. In the end, the
    eigenvalues of $\tB$ are $\tLambda$ and the diagonal entries of $\tB$
    are the same as that of $A$.

    The underlying undirected graph of $A$ is a connected graph since $A$
    is symmetric and irreducible, where the graph is built according to
    the non-zero pattern of $A$. We construct a spanning tree $T$ covering
    all vertices in the graph. Starting from a leaf vertex $v_i$ in $T$,
    assume the parent vertex of $v_i$ is $v_j$. Taking the $2\times 2$
    principal submatrix of $B$ at the intersections of the $i$-th and
    $j$-th columns and rows, we obtain a symmetric matrix as in
    \cref{lem:givens},
    \begin{equation*}
        B^{(ij)} =
        \begin{bmatrix}
            B_{ii} & B_{ij} \\
            B_{ji} & B_{jj} \\
        \end{bmatrix},
    \end{equation*}
    where $B_{ii}$ is always in the $(1,1)$ entry of $B^{(ij)}$. The
    off-diagonal entries of $B^{(ij)}$ are nonzero and $\Theta(1)$ with
    respect to $\veps$. By the first scenario in \Cref{lem:givens}, there
    exists a Givens rotation matrix $G^{(ij)}$ such that the $(1,1)$ entry
    of $G^{(ij)} B^{(ij)} (G^{(ij)})^\top$ is the target diagonal value
    $d_i$. Also, $G^{(ij)}$ is close to an identity matrix when $\veps$ is
    sufficiently small, i.e., $\fnorm{G^{(ij)}-I} =
    \Theta(\veps^{\alpha_1})$ with $\alpha_1 \geq 1$. Embed $G^{(ij)}$
    into an $n \times n$ Givens rotation matrix $G_1$ with entries at the
    intersections of the $i$-th and $j$-th columns and rows. The $(i,i)$
    entry of $G_1 B G_1^\top$ is again $d_i$. Importantly, when $\veps$ is
    sufficiently small, this operation only changes $(i,i)$ and $(j,j)$
    entries along the diagonal of $B$ and preserves the connectivity in
    the graph and, hence, the spanning tree.~\footnote{New edges could be
    added to the graph. However, the spanning tree is still a spanning
    tree in the updated graph.} After this step, the $i$-th diagonal entry
    has been ``corrected'' and all later operations will not touch it
    anymore. Hence, we can see that the vertex $v_i$ has been eliminated
    from the graph and tree. Then, we repeat this process with the updated
    tree and the matrix $G_1 B G_1^\top$. Such a process could be repeated
    $n-1$ times and results $G_1, \ldots, G_{n-1}$ Givens rotation
    matrices. Note that the perturbation order of the diagonals is always
    $O(\veps)$ during the procedure, we have  
    $\fnorm{G_i - I}=\Theta(\veps^{\alpha_i})$ with $\alpha_i \geq 1$
    for $i =1, \ldots, n-1$. Finally, we obtain a symmetric matrix,
    \begin{equation*}
        \tB = G_{n-1} \cdots G_1 B G_1^\top \cdots G_{n-1}^\top,
    \end{equation*}
    whose $n-1$ diagonal entries are ``corrected'' during the process
    and all the perturbations are collected at the last diagonal entry,
    which is automatically corrected due to the last majorization
    relation. Eigenvalues of $\tB$ remain the same as $B$ during the
    similarity transformations. The difference between $\tB$ and $A$
    could be bounded as,
    \begin{equation*}
        \fnorm{\tB - A} \leq \fnorm{\tB - B} + \fnorm{B - A} = O(\veps).
    \end{equation*}
    Thus, $A$ is strongly Schur-Horn continuous.
\end{proof}

\begin{lemma} \label{lem:two_strongSH}
    Given $A_1\in \bbR^{n_1\times n_1}$ and $A_2\in \bbR^{n_2 \times n_2}$
    be two strongly Schur-Horn continuous matrices such that their
    spectrum windows have a nonzero measured intersection, i.e., 
    \begin{equation*}
        \mu(\omega(A_1) \cap \omega(A_2)) > 0,
    \end{equation*} 
    where $\mu(\cdot)$ is the Lebesgue measure over $\bbR$, then matrix $A
    = \begin{bmatrix} A_1 & \\ & A_2 \end{bmatrix}$ is also strongly
    Schur-Horn continuous.
\end{lemma}

\begin{proof}    
    Denote the eigendecomposition of $A_i$ as $A_i=Q_i \Lambda_i Q_i
    ^\top$ for $i=1,2$, where $Q_i$ is the eigenvector matrix and
    $\Lambda_i$ is the diagonal eigenvalue matrix, respectively. Given a
    perturbation of eigenvalues of $A$, denoted as $\tLambda =
    \begin{bmatrix} \tLambda_1 & \\ & \tLambda_2 \end{bmatrix}$ with
    $\tr(\tLambda_1) = \tr(\Lambda_1) + h(\veps)$ and $\tr(\tLambda_2) =
    \tr(\Lambda_2) - h(\veps)$. Since $\fnorm{\Lambda -
    \tLambda}=O(\veps)$, if $h(\veps)\neq 0$, we must have
    $h(\veps)=O(\veps)=\Theta(\veps^{\alpha})$ with $\alpha \geq 1$.

    Without loss of generality, we assume that $\lmin(A_1) \leq
    \lmin(A_2)$. By the assumption $\mu(\omega(A_1) \cap \omega(A_2)) >
    0$, spectrum windows of $A_1$ and $A_2$ admit either of the following
    cases,
    \begin{equation*}
        \lmin(A_1) \leq \lmin(A_2) < \lmax(A_1) \leq \lmax(A_2),
    \end{equation*}
    or
    \begin{equation*}
        \lmin(A_1) \leq \lmin(A_2) < \lmax(A_2) \leq \lmax(A_1).
    \end{equation*}
    In both cases, we have 
    \begin{equation*}
        \lmin(A_1) < \lmax(A_2) \quad \text{and}
        \quad \lmin(A_2) < \lmax(A_1),
    \end{equation*}
    and hence,
    \begin{equation}
        \label{eq:key-ineq-two-strongSH}
        \lmin(\tLambda_1) < \lmax(\tLambda_2) \quad \text{and}
        \quad \lmin(\tLambda_2) < \lmax(\tLambda_1),
    \end{equation}
    when $\veps$ is sufficiently small. Next, we split the discussion
    based on the sign of $h(\veps)$.
    
    If $h(\veps) > 0$, by \cref{lem:givens}, one can apply a Givens
    rotation between $\lmax(\tLambda_1)$ and $\lmin(\tLambda_2)$ to
    compensate $-h(\veps)$,
    \begin{equation*}
        \begin{bmatrix}
            \lmax(\tLambda_1) & 0 \\
            0 & \lmin(\tLambda_2)
        \end{bmatrix} \rightarrow
        \begin{bmatrix}
            \lmax(\tLambda_1) - h(\veps) & * \\
            * & \lmin(\tLambda_2) + h(\veps)
        \end{bmatrix},
    \end{equation*}
    which is the second scenario in \cref{tab:givens}. The rotation is a
    perturbation of an identity whose rotation angle $\theta =
    \Theta(\veps^{\alpha/2}) = O(\veps^{1/2})$ since $\alpha \geq 1$. We
    embed the 2-by-2 rotation matrix into a rotation matrix $G_0$ of the
    same size as $A$, which is also a perturbation of an identity and
    $\fnorm{G_0 - I}=O(\veps^{1/2})$.

    If $h(\veps) < 0$, by \cref{lem:givens}, one can apply a Givens
    rotation between $\lmin(\tLambda_1)$ and $\lmax(\tLambda_2)$ to
    compensate $-h(\veps)$,
    \begin{equation*}
        \begin{bmatrix}
            \lmin(\tLambda_1) & 0 \\
            0 & \lmax(\tLambda_2)
        \end{bmatrix} \rightarrow
        \begin{bmatrix}
            \lmin(\tLambda_1) - h(\veps) & * \\
            * & \lmax(\tLambda_2) + h(\veps)
        \end{bmatrix},
    \end{equation*}
    which is the second scenario in \cref{tab:givens}. The rotation is a
    perturbation of an identity whose rotation angle $\theta =
    \Theta(\veps^{\alpha/2}) = O(\veps^{1/2})$ since $\alpha \geq 1$. We
    embed the 2-by-2 rotation matrix into a rotation matrix $G_0$ of the
    same size as $A$, which is also a perturbation of an identity and
    $\fnorm{G_0 - I} = O(\veps^{1/2})$.
    
    If $h(\veps) = 0$, we simply set $G_0 = I$ and hence $\fnorm{G_0 - I}
    = O(\veps^{1/2})$.

    After the diagonal compensation, we denote the block form of $ G_0
    \tLambda G_0^\top$ as,
    \begin{equation*}
        G_0 \tLambda G_0^\top = 
        \begin{bmatrix}
            \Lambda_1' & E_{12} \\ E_{21} & \Lambda_2'
        \end{bmatrix}.
    \end{equation*}
    For diagonal blocks, $\Lambda_1'$ and $\Lambda_2'$ are diagonal
    matrices being $O(\veps)$ perturbation of $\Lambda_1$ and $\Lambda_2$
    respectively, satisfying $\tr(\Lambda_1') = \tr(\Lambda_1) = \tr(A_1)$
    and $\tr(\Lambda_2') = \tr(\Lambda_2) = \tr(A_2)$. Off-diagonal blocks
    $E_{12}$ and $E_{21}$ are $O(\veps^{1/2})$ perturbations of zero
    matrices.

    Since $A_1$ and $A_2$ are strongly Schur-Horn continuous, there
    exists
    \begin{equation*}
        B_1 = G_{12} Q_1 G_{11} \Lambda_1' G_{11}^\top Q_1^\top G_{12}^\top
        \quad \text{and} \quad
        B_2 = G_{22} Q_2 G_{21} \Lambda_2' G_{21}^\top Q_2^\top G_{22}^\top
    \end{equation*}
    such that $\diag(B_i)=\diag(A_i)$ and $G_{ij}$s being
    $O(\veps^{1/2})$ perturbation of identity matrices with
    $\fnorm{G_{ij} - I}=O(\veps^{1/2})$ for $i, j=1, 2$. Introducing
    \begin{equation*}
        G_1 = \begin{bmatrix}
            G_{11} &  \\  & G_{21}
        \end{bmatrix} G_0, \quad 
        Q = \begin{bmatrix}
            Q_1 & \\  & Q_2
        \end{bmatrix}, \quad 
        G_2 = \begin{bmatrix}
            G_{12} & \\ & G_{22}
        \end{bmatrix},
    \end{equation*}
    we obtain
    \begin{equation*}
        \tB = G_2 Q G_1 \tLambda G_1^\top Q^\top G_2^\top,
    \end{equation*}
    which has eigenvalues being $\tLambda$ and diagonal entries being the
    same as that of $A$. Since all $G_{ij}$s and $G_0$ are orthogonal
    matrices close to identity matrices in order $O(\veps^{1/2})$
    under the Frobenius norm, $G_1$ and $G_2$ are also orthogonal matrices
    being $O(\veps^{1/2})$ perturbations of identity matrices, i.e.,
    $\fnorm{G_i - I}=O(\veps^{1/2})$ for $i=1, 2$. Hence, we conclude that
    $A$ is strongly Schur-Horn continuous.
\end{proof}

\begin{lemma} \label{lem:strongSH_diag}
    Let $A_1\in \bbR^{n\times n}$ be a strongly Schur-Horn continuous
    matrix with spectrum window such that $\mu(\omega(A_1)) > 0$, then
    for any $d_2\in \omega(A_1)^{\circ}$, matrix $A = \begin{bmatrix}
    A_1 & \\ & d_2 \end{bmatrix}$ is strongly Schur-Horn continuous.
\end{lemma}

The proof of \cref{lem:strongSH_diag} is similar to that of
\cref{lem:two_strongSH}. Note that from $d_2 \in \omega(A_1)^{\circ}$ we
have $\lmin(A_1) < d_2 < \lmax(A_1)$, which is the analogy of
\eqref{eq:key-ineq-two-strongSH} in the proof of \cref{lem:two_strongSH}.
Once we obtain these inequalities, the rest of the proofs are identical.
Hence, we omit the detail for simplicity.

\begin{lemma} \label{lem:block_diag_sSH}
    Let $A\in \bbR^{n\times n}$ be a symmetric matrix. There exists a
    permutation matrix $P$ such that the permuted matrix $P A P^\top$
    admits a block diagonal structure,
    \begin{equation*}
        P A P^\top =
        \begin{bmatrix}
            A_1 & & \\
            & \ddots & \\
            & & A_p
        \end{bmatrix},
    \end{equation*}
    and diagonal blocks $\{A_i\}_{i=1}^p$ are either scalars or
    non-scalar strongly Schur-Horn continuous matrices. Their
    eigenvalues are ordered, i.e.,
    \begin{equation} \label{eq:lem_block_diag_sSH}
        \lmin(A_1) \leq \lmax(A_1) \leq \lmin(A_2) \leq \lmax(A_2) \leq
        \cdots \leq \lmin(A_p) \leq \lmax(A_p).
    \end{equation}
\end{lemma}

\begin{proof}
    Given a symmetric matrix $A$, let $q$ be the number of irreducible
    submatrices in $A$ (including scalars). We sort these submatrices in
    ascending order with their smallest eigenvalue being the first index
    and their largest eigenvalues being the second index. Hence, there
    exists a permutation matrix $P$ such that,
    \begin{equation*}
        P A P^\top =
        \begin{bmatrix}
            B_1 & & \\
            & \ddots & \\
            & & B_q
        \end{bmatrix},
    \end{equation*}
    and the smallest and largest eigenvalues of $i$-th and $j$-th blocks
    for $i < j$ admit either,
    \begin{enumerate}[(i)]
        \item $\lmin(B_i) < \lmin(B_j)$; or
        \item $\lmin(B_i) = \lmin(B_j)$ and $\lmax(B_i) \leq \lmax(B_j)$.
    \end{enumerate}
    Following these two eigenvalue conditions, we obtain the inequality
    immediately,
    \begin{equation*}
        \lmin(\diag(B_i, B_{i+1}, \dots, B_q)) = \lmin(B_i).
    \end{equation*}
    Next, we describe a procedure to collect contiguous
    $B_i$s and denote them as $A_j$ such that the smallest and largest
    eigenvalues of $A_j$s are ordered.

    We construct $A_1$ step-by-step and start with $A_1 = B_1$. By the
    construction of $P A P^\top$, we know that $B_1$ is either a diagonal
    matrix or an irreducible matrix. If $B_1$ is an irreducible matrix,
    then by \cref{lem:irreducible_case}, $B_1$ is strongly Schur-Horn
    continuous. Hence, $A_1 = B_1$ is either a scalar or a strongly
    Schur-Horn continuous matrix.
    
    We first consider the case that $A_1$ is a scalar and $A_1$ by itself
    is a block. Combined with the construction of $B_i$s, we have,
    \begin{equation} \label{eq:blocksSH-ineq-diag}
        \lmax(A_1) = B_1 \leq \lmin(B_2)
        = \lmin(\diag(B_2, \dots, B_q)).
    \end{equation}

    The second case is that $A_1$ is a strongly Schur-Horn continuous
    matrix. Note that according to \cref{lem:irreducible_eig_gap},
    irreducible matrix $B_1$ satisfies that $\mu(\omega(B_1))>0$. In
    this case, we expand $A_1 = \diag(B_1, \dots, B_i)$ only if either
    $B_i$ is an irreducible matrix and $\mu(\omega(A_1) \cap
    \omega(B_i)) > 0$; or $B_i$ is a scalar and $B_i \in
    \omega(A_1)^\circ$. By \cref{lem:two_strongSH} and
    \cref{lem:strongSH_diag} respectively, the expanded matrix $A_1 =
    \diag(B_1, \dots, B_i)$ is strongly Schur-Horn continuous. If the
    expansion terminates, we have $\mu(\omega(A_1) \cap \omega(B_i)) =
    0$ when $B_i$ is an irreducible matrix, and $\lmax(A_1) \leq B_i$
    when $B_i$ is a scalar. Combined with the construction of $A_1$ and
    $B_i$s, we have,
    \begin{equation} \label{eq:blocksSH-ineq-sSH}
        \lmax(A_1) \leq \lmin(B_i)
        = \lmin(\diag(B_i, \dots, B_q)),
    \end{equation}
    with $A_1=\diag(B_1, \ldots, B_{i-1})$ and $i > 1$. 

    When the expansion of $A_1$ terminates at block $B_i$, the
    construction of $A_2$ starts from $B_i$ following the same procedure
    as above. Eventually, we obtain $A_1, \dots, A_p$ and the inequality
    \eqref{eq:lem_block_diag_sSH} follows from
    \eqref{eq:blocksSH-ineq-diag} and \eqref{eq:blocksSH-ineq-sSH}
    directly.
\end{proof}

\begin{proposition}
    If $A\in \bbR^{n\times n}$ is a symmetric matrix satisfying the
    strict majorization condition \eqref{eq:strict-major}, then $A$ is
    strongly Schur-Horn continuous.
\end{proposition}

\begin{proof}
According to \cref{lem:block_diag_sSH}, after a permutation $P$, the
matrix admits a block diagonal structure, i.e., $P A P^\top =
\diag\left(A_1, \cdots, A_p\right)$, such that their eigenvalues are
ordered. The majorization conditions hold for every symmetric block $A_i$
and we have $\lmin(A_i) \leq \dmin(A_i)$ and $\dmax(A_i) \leq
\lmax(A_i)$.\footnote{Notations $\dmin(\cdot)$ and $\dmax(\cdot)$ denote
the minimum and maximum of diagonal entries of a matrix, respectively.}
Combined with \eqref{eq:lem_block_diag_sSH}, we have 
\begin{equation*}
    \dmin(A_1) \leq \dmax(A_1) \leq \dmin(A_2) \leq \dmax(A_2) \leq
    \cdots \leq \dmin(A_p) \leq \dmax(A_p).
\end{equation*}
Since $P A P^\top$ is a block diagonal matrix, we have
\begin{align*}
    \tr(\Lambda_1) & = \tr(A_1), \\
    \tr(\Lambda_1) + \tr(\Lambda_2) & = \tr(A_1) + \tr(A_2), \\
     & ~\vdots \\
    \tr(\Lambda_1) + \cdots + \tr(\Lambda_{p-1}) &
    = \tr(A_1) + \cdots + \tr(A_{p-1}), \\
    \tr(\Lambda_1) + \cdots + \tr(\Lambda_{p-1}) + \tr(\Lambda_p) &
    = \tr(A_1) + \cdots + \tr(A_{p-1}) + \tr(A_p),
\end{align*}
where $\Lambda_1, \ldots, \Lambda_p$ are diagonal eigenvalue submatrices
corresponding to the block structures in $PAP^\top$. The relations above
contradict with the strict majorization condtion \eqref{eq:strict-major}
when $p \geq 2$. Hence, there is only one block in matrix $A$, which is
either a scalar or non-scalar strongly Schur-Horn continuous. Note that a
scalar is also strongly Schur-Horn continuous. Thus matrix $A$ must be
strongly Schur-Horn continuous. 
\end{proof}

\section{Schur-Horn Continuity of Symmetric Matrices}
\label{sec:SHC_symmetric}

Now we turn to the proof of our main result, \cref{thm:continuitySH}.

\begin{proof}
    Our proof is similar to that of \cref{thm:diagonal_case}, where
    diagonal scalars are replaced by either scalars or strongly Schur-Horn
    continuous blocks.

    By \cref{lem:block_diag_sSH}, the symmetric matrix $A$ admits a block
    diagonal form after a permutation. Without loss of generality, we
    assume $A$ is a block diagonal matrix, $A = \diag(A_1,
    \cdots, A_p)$, where $A_i$ is either a scalar or a strongly Schur-Horn
    continuous matrix. The smallest and largest eigenvalues of these
    blocks are ordered as in \eqref{eq:lem_block_diag_sSH}. The
    majorization conditions hold for every symmetric block $A_i$ and we
    have $\lmin(A_i) \leq \dmin(A_i)$ and $\dmax(A_i) \leq \lmax(A_i)$.
    Combined with \eqref{eq:lem_block_diag_sSH}, we have 
    \begin{equation*}
        \dmin(A_1) \leq \dmax(A_1) \leq \dmin(A_2) \leq \dmax(A_2) \leq
        \cdots \leq \dmin(A_p) \leq \dmax(A_p).
    \end{equation*}

    The perturbed eigenvalues are denoted as $\tlambda_i$. When the
    eigenvalues have a gap, i.e., $\lambda_i < \lambda_{i+1}$, the
    perturbed eigenvalues keep the ordering, i.e., $\tlambda_i \leq
    \tlambda_{i+1}$ for sufficiently small $\veps$. When eigenvalues are
    identical, $\lambda_i = \lambda_{i+1}$, the perturbed eigenvalues are
    ordered based on their perturbations. In another point of view, we
    could regard the perturbations on identical eigenvalues as sorted
    perturbations. Therefore, for $\veps>0$ sufficiently small, we could
    obtain the block-wise majorization conditions,
    \begin{align*}
        \tr(\tLambda_1) & \leq \tr(A_1), \\
        \tr(\tLambda_1) + \tr(\tLambda_2) & \leq \tr(A_1) + \tr(A_2), \\
         & ~\vdots \\
        \tr(\tLambda_1) + \cdots + \tr(\tLambda_{p-1}) &
        \leq \tr(A_1) + \cdots + \tr(A_{p-1}), \\
        \tr(\tLambda_1) + \cdots + \tr(\tLambda_{p-1}) + \tr(\tLambda_p) &
        = \tr(A_1) + \cdots + \tr(A_{p-1}) + \tr(A_p),
    \end{align*}
    where $\tLambda_1, \ldots, \tLambda_p$ are perturbed diagonal
    eigenvalue submatrices corresponding to the block structure as in $A$.
    We denote the block-wise perturbations as $h_i(\veps) =
    \tr(\tLambda_i) - \tr(A_i)$, and similarly have,
    \begin{equation*}
        \begin{split}
            & h_1(\veps) + h_2(\veps) + \cdots + h_i(\veps) \leq 0, \quad
            i = 1, 2, \dots, p-1, \text{ and} \\
            & h_1(\veps) + \cdots + h_p(\veps) = 0.
        \end{split}
    \end{equation*}

    We also maintain a priority queue with diagonal block indices as
    elements. For any diagonal block index $i$ in the queue, we ensure
    that $\th_i(\veps)$ is a negative perturbation $\th_i(\veps) < 0$,
    where $\th_i(\veps)$ denotes the updated perturbation throughout the
    procedure. Starting from the first diagonal block, we check and
    enqueue the index $i = 1, 2, \dots$ in order if $\th_i(\veps) < 0$ and
    skip the index $i$ if $\th_i(\veps) = 0$. We keep on checking and
    enqueuing indices until the first index $j$ such that $\th_j(\veps) >
    0$. If $j$ does not exist, then by the last equation in the block-wise
    majorization relation, we know that the queue is also empty and the
    diagonal blocks of the perturbed $A$ have all been corrected, i.e.,
    \begin{equation} \label{eq:block-correct-perturbation}
        \th_i(\veps) = 0, \quad i = 1, \dots, p.
    \end{equation}
    Otherwise, we obtain a $j$ and the updated perturbations satisfy,
    \begin{equation} \label{eq:block-updated-majorization}
        \th_1(\veps) + \cdots + \th_j(\veps) = h_1(\veps) + \cdots +
        h_j(\veps) \leq 0.
    \end{equation}
    This condition is satisfied in the first step and we will verify it
    after each step. By the $j$-th block-wise majorization relation, the
    queue is guaranteed to be non-empty. We pop an index from the queue
    and denote it as $i$. There are two scenarios here based on the
    block size of $A_i$ and $A_j$.

    \begin{enumerate}[(i)]
        \item Both $A_i$ and $A_j$ are one dimensional, i.e., $A_i = d_i$
        and $A_j = d_j$. This scenario is the same as
        \cref{thm:diagonal_case} for diagonal matrix and we have
        \eqref{eq:rotate_diag_case}. Thus, if $d_i < d_j$, the Givens
        rotation itself is close to identity; if $d_i = d_j$, while the
        Givens rotation could be away from identity, the rotated matrix is
        close to the original block diagonal matrix.   

        \item At least one of $A_i$ and $A_j$ is non-diagonal strongly
        Schur-Horn continuous matrix. By \cref{prop:strongSH-major}, we
        have either $\lmin(A_i) < \lmax(A_i)$ or $\lmin(A_j) <
        \lmax(A_j)$, or both. Combined with the ordering of eigenvalues,
        \eqref{eq:lem_block_diag_sSH}, the inequality $\lmin(A_i) <
        \lmax(A_j)$ holds. Thus, by \cref{lem:givens}, we can perform a
        Givens rotation close to the identity between $\lmin(A_i)$ and
        $\lmax(A_j)$ to revert $-\th_i(\veps)$, making the sum of the $i$-th
        block perturbations equals zero, i.e.,
        \begin{equation*}
            \begin{bmatrix}
                \lmin(\tLambda_i) & 0 \\
                0 & \lmax(\tLambda_j)
            \end{bmatrix} \rightarrow
            \begin{bmatrix}
                \lmin(\tLambda_i) - \th_i(\veps) & * \\
                * & \lmax(\tLambda_j) + \th_i(\veps)
            \end{bmatrix}.
        \end{equation*}
    \end{enumerate}

    We now have three cases: i) $\th_i(\veps) + \th_j(\veps) < 0$; ii)
    $\th_i(\veps) + \th_j(\veps) = 0$; and iii) $\th_i(\veps) +
    \th_j(\veps) > 0$. In case i), we enqueue $j$ and start checking the
    following indices. In case ii), we skip $j$ and start checking the
    indices after $j$. In case iii), we pop another index from the queue
    and repeat the correction procedure. In all cases, the updated total
    perturbations at the $i$-th block and the $j$-th block are $0$ and
    $\th_i(\veps) + \th_j(\veps)$, respectively. Hence
    \eqref{eq:block-updated-majorization} holds for all indices greater
    or equal to $j$. Then, the majorization relations guarantee that the
    procedure ends if and only if all diagonal blocks have been
    corrected, i.e., equations \eqref{eq:block-correct-perturbation}
    hold. Notice that all above corrections are on the diagonal
    eigenvalue matrix $\tLambda$. The corrected matrix $\tLambda^{(p)}$
    admits form,
    \begin{equation*}
        \tLambda^{(p)} =
        \begin{bmatrix}
            \tLambda_1^{(p)} & \cdots & * \\
            \vdots & \ddots & \vdots \\
            * & \cdots & \tLambda_p^{(p)}
        \end{bmatrix},
    \end{equation*}
    with $\tLambda_i^{(p)}$ being diagonal matrices satisfying
    $\tr(\tLambda_i^{(p)})=\tr(A_i)$ for $i = 1, \dots, p$. The diagonals
    of $\tLambda^{(p)}$ are always $O(\veps)$ perturbation to diagonals of
    $\Lambda$ during the trace adjustment between diagonal blocks, i.e.,
    \begin{equation*}
        \fnorm{\tLambda_i^{(p)} - \Lambda_i}=O(\veps), \quad i=1,\ldots,p.
    \end{equation*}
    Furthermore, analog to the proof of \cref{thm:diagonal_case}, it
    yields that all the off-diagonals of $\tLambda^{(p)}$ are
    $O(\veps^{1/2})$. Thus, the distance between $\tLambda$ and
    $\tLambda^{(p)}$ obeys,
    \begin{equation} \label{eq:tLambda-block-estimate}
        \fnorm{\tLambda - \tLambda^{(p)}} = O(\veps^{1/2}).
    \end{equation}

    If the $i$-th diagonal block is a scalar, the correction procedure
    guarantees the diagonal entry is correct. If the $i$-th diagonal block
    is a strongly Schur-Horn continuous matrix, then the
    \cref{def:strongSHcont} ensures the existence of $G_{i1}$ and $G_{i2}$
    close to identity and $\tB_i = G_{i2} Q_i G_{i1} \tLambda_i^{(p)}
    G_{i1}^\top Q_i^\top G_{i2}^\top$ has the same diagonal entries as
    $A_i$, where $Q_i$ is the eigenvector matrix of $A_i$. Assembling
    $\{Q_i\}$, $\{G_{i1}\}$, and $\{G_{i2}\}$ together, we denote them as,
    \begin{equation*}
        G_1 = \begin{bmatrix}
            G_{11} &  &  \\
             & \ddots & \\
             &  & G_{p1}
        \end{bmatrix}, \quad 
        Q = \begin{bmatrix}
            Q_{1} &  &  \\
             & \ddots & \\
             &  & Q_{p}
        \end{bmatrix}, \quad 
        G_2 = \begin{bmatrix}
            G_{12} &  &  \\
             & \ddots & \\
             &  & G_{p2}
        \end{bmatrix},
    \end{equation*}
    where $G_{i1} = Q_i = G_{i2} = 1$ if $A_i$ is a scalar. By
    construction, we know that $G_1$ and $G_2$ are close to the identity
    matrix with $\fnorm{G_i - I}=O(\veps^{1/2})$ for $i=1, 2$. The matrix
    $\tB = G_2 Q G_1 \tLambda^{(p)} G_1^\top Q^\top G_2^\top$ has the same
    diagonal entries as $A$ and eigenvalues being $\tLambda$. We verify
    that $\tB$ is close to $A$,
    \begin{equation*}
        \begin{split}
            \fnorm{\tB - A} &
            = \fnorm{G_2 Q G_1 \tLambda^{(p)} G_1^\top Q^\top G_2^\top
            - Q \Lambda Q^\top} \\
            & \leq \fnorm{Q G_1 \tLambda^{(p)} G_1^\top Q^\top
            - Q \Lambda Q^\top} + O(\veps^{1/2}) \\
            & = \fnorm{G_1 \tLambda^{(p)} G_1^\top
            - \Lambda} + O(\veps^{1/2}) \\
            & \leq \fnorm{\tLambda^{(p)}
            - \Lambda} + O(\veps^{1/2}) \\
            & \leq \fnorm{\tLambda^{(p)}
            - \tLambda} + \fnorm{\tLambda - \Lambda} + O(\veps^{1/2}) = O(\veps^{1/2}),
        \end{split}
    \end{equation*}
    where the first and second inequalities are due to the fact that $G_1$
    and $G_2$ are close to the identity matrix with $\fnorm{G_i -
    I}=O(\veps^{1/2})$ for $i=1, 2$, the second equality is due to the
    unitary invariant property of the Frobenius norm, the last inequality
    is due to the norm triangular inequality, and the last equality is due
    to the definition of $\tLambda$ and \eqref{eq:tLambda-block-estimate}.
\end{proof}

\section{Schur-Horn Continuity of Hermitian Matrices} \label{sec:SHC_Hermitian}

Next, we discuss the Schur-Horn continuity of a Hermitian matrix, as
described in \cref{thm:continuitySH_hermit}. The proof follows a similar
approach as that for symmetric matrices, with an extra step to
generalize \cref{lem:givens} to its complex counterpart.

Denote $\ri=\sqrt{-1}$, we introduce the complex Givens rotation defined as   
\begin{equation} \label{eq:comp_givens}
    G=\begin{bmatrix}
        \re^{\ri\phi}\cos\theta & \re^{\ri\psi}\sin\theta \\
        -\re^{-\ri\psi}\sin\theta & \re^{-\ri\phi}\cos\theta
    \end{bmatrix}            
\end{equation}
where $\theta, \phi, \psi \in \bbR$.

\begin{lemma} \label{lem:givens_complex}
    Given $\veps >0$ small enough and $d_1, d_2 \in \bbR$. Let a
    Hermitian matrix $B$ of form,
    \begin{equation*}
        B = \begin{bmatrix}
            b_{11} & b_{12} \\
            b_{21} & b_{22}
        \end{bmatrix} = \begin{bmatrix}
            d_1 - f(\veps) & b_{12} \\
            b_{12}^* & d_2 + g(\veps)
        \end{bmatrix},
    \end{equation*}
    with $f(\veps) = \Theta(\veps^\alpha),
    g(\veps)=\Theta(\veps^{\beta})$ for $\alpha, \beta > 0$. Further, we
    assume that
    \begin{equation*}
        |b_{12}|^2 + f(\veps)(d_2 - d_1 + g(\veps)) \geq 0.
    \end{equation*}
    Then there exists a complex Givens rotation $G$ with rotation angle
    $\theta = \Theta(\veps^\gamma)$ and $\phi, \psi \in \bbR$ such that
    the $(1, 1)$ entry of $\tB=GBG^*$ is $\tilde{b}_{11}=d_1$ and
    $\fnorm{\tB - B}=O(\veps^{\delta})$ where various scenarios of
    $\gamma$ and $\delta$ are provided in \cref{tab:givens_complex}.
    
    \begin{table}[htbp]
        \centering
        \begin{tabular}{lllcc}
            \toprule
            \multicolumn{3}{c}{Various Scenarios} & $\gamma$ & $\delta$ \\
            \toprule
            $b_{12} \neq 0$ & & & $\alpha$ & $\alpha$ \\
            \midrule
            $b_{12} = 0$ & $d_1 \neq d_2$ & & $\alpha / 2$ & $\alpha / 2$ \\
            \midrule
            $b_{12} = 0$ & $d_1 = d_2$ & $\alpha > \beta$ & $(\alpha-\beta)/2$ & $(\alpha+\beta)/2$ \\
            \midrule
            $b_{12} = 0$ & $d_1 = d_2$ & $\alpha \leq \beta$ & $0$ & $\alpha$ \\
            \bottomrule
        \end{tabular}
        \caption{Various scenarios of $b_{12}$, $d_1$, $d_2$, $\gamma$, and
        $\delta$ for \cref{lem:givens_complex}.}
        \label{tab:givens_complex}
    \end{table}
\end{lemma}

\begin{proof}
    Note that $B$ is a Hermitian matrix whose diagonal entries are all
    real numbers, if $b_{12}=0$, it reduces to \cref{lem:givens} and we
    have the conclusion for such scenarios directly. Below we assume
    that $b_{12}\neq 0$. Denote $b_{12}= \rep(b_{12}) +
    \ri\,\imp(b_{12})$ where $\rep(b_{12}) \in\bbR$ and $\imp(b_{12})\in
    \bbR$ are real part and imaginary part of $b_{12}$, respectively. We
    further split the discussion into two scenarios: (i) $\rep(b_{12})
    \neq 0$ and (ii) $\rep(b_{12}) = 0$.   

    First, if $\rep(b_{12}) \neq 0$, we still apply the real Givens
    rotation denoted as $G=\begin{bmatrix} c & s \\ -s & c \end{bmatrix}$ 
    with $c=\cos\theta$ and $s=\sin\theta$ to matrix $B$ and obtain, 
    \begin{equation*}
        G B G^* = \begin{bmatrix}
            c^2b_{11} + s^2b_{22} + 2cs\rep(b_{12}) & \omega \\
            \omega^* & c^2b_{22} + s^2b_{11} - 2cs\rep(b_{12}) 
        \end{bmatrix},
    \end{equation*}
    with $\omega=cs(b_{22}-b_{11}) + (c^2-s^2)\rep(b_{12}) +
    \ri\,\imp(b_{12})$. Equating the $(1,1)$ entry of $\tB=GBG^*$ and
    $d_1$ it leads to 
    \begin{equation*}
        c^2 b_{11} + s^2 b_{22} + 2cs \rep(b_{12}) = d_1, 
    \end{equation*}
    which is analogous to \eqref{eq:secondorderoriginal} in the proof of
    \cref{lem:givens} with $b_{12}$ replaced by $\rep(b_{12})$. Note
    that $\rep(b_{12})\neq 0$, adopting a similar analysis one concludes
    that $\gamma=\alpha$ and $\delta=\alpha$. 

    Second, if $\rep(b_{12})=0$, at this time from $b_{12}\neq 0$ we
    must have $\imp(b_{12})\neq 0$. Thus we consider another Givens
    rotation matrix denoted as $G=\begin{bmatrix} \ri c & s \\ -s & -\ri
    c \end{bmatrix}$, with $\phi=\frac{\pi}{2}$ and $\psi=0$. Applying
    $G$ from the left of $B$ and $G^*$ from the right of $B$ we get,
    \begin{equation*}
        G B G^* = \begin{bmatrix}
            c^2b_{11} + s^2b_{22} - 2cs\imp(b_{12}) & \omega \\
            \omega^* & c^2b_{22} + s^2b_{11} + 2cs\imp(b_{12}) 
        \end{bmatrix},
    \end{equation*}
    with $\omega=\ri cs(b_{22}-b_{11}) - \ri (c^2-s^2)\imp(b_{12}) -
    \rep(b_{12})$.   
    Equating the $(1,1)$ entry of $\tB=GBG^*$ and $d_1$ it leads to 
    \begin{equation*}
        c^2 b_{11} + s^2 b_{22} - 2cs \imp(b_{12}) = d_1,
    \end{equation*}
    which is analogous to \eqref{eq:secondorderoriginal} in the proof of
    \cref{lem:givens} with $b_{12}$ replaced by $-\imp(b_{12})$. Note that
    $\imp(b_{12})\neq 0$, adopting a similar analysis one concludes that
    $\gamma=\alpha$ and $\delta=\alpha$. 
\end{proof}

Now, we generalize the definition of strong Schur-Horn continuity for
symmetric matrices to Hermitian matrices, by simply replacing orthogonal
matrices in \cref{def:strongSHcont} with unitary matrices.

\begin{definition}[Strong Schur-Horn Continuity for Hermitian Matrices]
    Suppose $A \in \bbC^{n \times n}$ is a Hermitian matrix with an
    eigendecomposition $A = Q \Lambda Q^*$, where $Q$ is the unitary
    eigenvector matrix and $\Lambda$ is the diagonal eigenvalue matrix.
    Matrix $A$ is strongly Schur-Horn continuous if, for any perturbed
    eigenvalues $\tLambda$ satisfying $\tr(\tLambda) = \tr(\Lambda)$ and
    $\fnorm{\tLambda - \Lambda} = O(\veps)$ for $\veps > 0$ sufficiently
    small, there exists a Hermitian matrix $\tB = G_2 Q G_1 \tLambda
    G_1^* Q^* G_2^*$ such that
    \begin{enumerate}
        \item $\diag(\tB)=\diag(A)$,
        \item $G_1$ and $G_2$ are unitary matrices, and
        \item $\fnorm{G_i - I} = O(\veps^{1/2})$ for
        $i = 1, 2$.
    \end{enumerate}
\end{definition}

With \cref{lem:givens_complex}, one can verify that the desired properties
of strong Schur-Horn continuity appeared in \cref{sec:strongSHcont} also
hold for Hermitian matrices. Thus, by employing a similar analysis, one
can prove \cref{thm:continuitySH_hermit} and establish the Schur-Horn
continuity of Hermitian matrices, and we omit the details.

We also have the proposition for strict majorization conditions and strong
Schur-Horn continuity of Hermitian matrices.

\begin{proposition} \label{prop:sSH_major_hermit}
    If $A\in \bbC^{n\times n}$ is a Hermitian matrix satisfying the strict
    majorization condition \eqref{eq:strict-major}, then $A$ is strongly
    Schur-Horn continuous.
\end{proposition}

\section{Conclusion} \label{sec:conclusion}

In this paper, we explore the eigenvalue perturbation of a symmetric
(Hermitian) matrix with fixed diagonals, which is referred to as the
continuity of the Schur-Horn mapping. We first establish the Schur-Horn
continuity for real diagonal matrices leveraging Givens rotation and
majorization relations between diagonals and eigenvalues. Then, we
introduce the concept of the strong Schur-Horn continuity, which is a
stronger version of Schur-Horn continuity. This allows us to construct a
block-wise majorization relation and prove the Schur-Horn continuity for
general symmetric matrices. Additionally, our analysis could be extended
to Hermitian matrices to establish their Schur-Horn continuity.

\section*{Acknowledgments}

This work is supported in part by the National Natural Science Foundation
of China (12271109) and Shanghai Pilot Program for Basic Research - Fudan
University 21TQ1400100 (22TQ017).

\appendix

\section{Proof of \cref{prop:strongSH-major}}
\label{app:strongSH-major-proof}

\begin{proof}
    We prove by contrapositive. If $A$ is not a scalar matrix and does not
    satisfy \eqref{eq:strict-major}, then we have either (i) $A$ not
    satisfying the majorization relation \eqref{eq:major}; or (ii) $A$
    satisfying the majorization relation, but there exists $1 \leq i, j <
    n$ such that,
    \begin{equation*}
        \lambda_1 + \cdots + \lambda_i = d_1 + \cdots + d_i,
        \quad \text{and} \quad
        \lambda_1 + \cdots + \lambda_j < d_1 + \cdots + d_j.
    \end{equation*}

    In case (i), $A$ does not satisfy the majorization relation. Hence,
    by Schur-Horn theorem, $A$ is not strongly Schur-Horn continuous.

    In case (ii), we have $\lambda_1 < \lambda_n$. The discussion is
    further split into two scenarios: (ii.1) $\lambda_i < \lambda_n$ and
    (ii.2) $\lambda_i = \lambda_n$.
    
    In (ii.1), we denote those eigenvalues equal to $\lambda_i$ and
    $\lambda_n$ as follows,
    \begin{equation*}
        \lambda_\ell < \lambda_{\ell+1} = \cdots = \lambda_i = \cdots
        = \lambda_r < \lambda_{r+1}, 
        \quad \text{and} \quad \lambda_k < \lambda_{k+1} = \cdots = \lambda_n,
    \end{equation*}
    where $0 \leq \ell < i \leq r \leq k < n$. Then we consider a
    particular perturbation by adding an $\veps > 0$ small enough to those
    eigenvalues equal to $\lambda_i$ and subtracting $\frac{r-\ell}{n-k}
    \veps$ from those eigenvalues equal to $\lambda_n$, i.e., 
    \begin{multline*}
        \tlambda = \Bigl( 
            \lambda_1, \cdots, \lambda_{\ell}, 
            \lambda_{\ell+1} + \veps, \cdots, \lambda_i + \veps, \cdots,
            \lambda_{r} + \veps, \\ 
            \lambda_{r+1}, \cdots, \lambda_{k},
            \lambda_{k+1} - \frac{r-\ell}{n-k}\veps, \cdots, \lambda_n
            - \frac{r-\ell}{n-k}\veps 
            \Bigl).
    \end{multline*}
    This perturbed $\tLambda = \diag(\tlambda)$ satisfies
    $\tr(\tLambda)=\tr(\Lambda)$ and $\fnorm{\tLambda - \Lambda}=O(\veps)$.
    Furthermore, $\tlambda$ is still in a non-decreasing order for
    $\veps>0$ sufficiently small. While, the $i$-th majorization relation
    between $\tlambda$ and $d$ is violated,
    \begin{equation*}
        \tlambda_1 + \cdots + \tlambda_i = \lambda_1 + \cdots + \lambda_i
        + (i-\ell)\veps = d_1 + \cdots + d_i + (i-\ell)\veps
        > d_1 + \cdots + d_i.
    \end{equation*}
    By Schur-Horn theorem, there does not exist a matrix with diagonal and
    eigenvalues being $d$ and $\tlambda$, respectively. Hence, $A$ is not
    strongly Schur-Horn continuous.
    
    In (ii.2), we have $\lambda_1 < \lambda_i = \lambda_n$. Denote those
    eigenvalues equal to $\lambda_1$ and $\lambda_i$ as follows,
    \begin{equation*}
        \lambda_1 = \cdots = \lambda_r < \lambda_{r+1} 
        \quad \text{and} \quad 
        \lambda_k < \lambda_{k+1} = \cdots = \lambda_i = \cdots = \lambda_n,
    \end{equation*}
    where $1 \leq r \leq k < n$. Consider a particular perturbation by
    subtracting $\veps > 0$ from those eigenvalues equal to $\lambda_i$
    and adding $\frac{n-k}{r}\veps$ to those eigenvalues equal to
    $\lambda_1$ as follows,
    \begin{equation*}
        \tlambda = 
        \left( \lambda_1 + \frac{n-k}{r}\veps, \cdots,
        \lambda_r + \frac{n-k}{r}\veps,  
        \lambda_{r+1}, \cdots, \lambda_k,
        \lambda_{k+1} - \veps, \cdots, 
        \lambda_n - \veps \right).
    \end{equation*}
    This perturbed $\tLambda = \diag(\tlambda)$ satisfies $\tr(\tLambda) =
    \tr(\Lambda)$ and $\fnorm{\tLambda - \Lambda}=O(\veps)$. Furthermore,
    $\tlambda$ is in a non-decreasing order when $\veps>0$ is sufficiently
    small. Similarly, the $i$-th majorization relation between $\tlambda$
    and $d$ is violated,
    \begin{multline*}
        \tlambda_1 + \cdots + \tlambda_i = \lambda_1 + \cdots + \lambda_i
        + r \cdot \frac{n-k}{r}\veps - (i-k) \veps \\
        = d_1 + \cdots + d_i + (n-i)\veps > d_1 + \cdots + d_i.
    \end{multline*}
    By Schur-Horn theorem, there does not exist a matrix with diagonal and
    eigenvalues being $d$ and $\tlambda$, respectively. Hence, $A$ is not
    strongly Schur-Horn continuous.
\end{proof}

\section{Spectrum Window for Irreducible Matrices}

\begin{lemma} \label{lem:irreducible_eig_gap}
    If $A$ is an irreducible symmetric matrix whose dimension is strictly
    greater than 1, then $\lmin(A) < \lmax(A)$. 
\end{lemma}

\begin{proof}
    Consider the eigendecomposition $A = Q \Lambda Q^\top$ of the
    symmetric matrix $A$, where $Q$ is the orthonormal matrix composed of
    the eigenvectors of $A$, and $\Lambda$ is the diagonal eigenvalue
    matrix. Suppose $\lmin(A) = \lmax(A) = \lambda$, then we have $A = Q
    \cdot \lambda I \cdot Q^\top = \lambda I$, which contradicts the
    irreducibility of $A$.
\end{proof}

\section{Application in qOMM Landscape Analysis}
\label{app:qomm_application}

In this section, we demonstrate the utility of Schur-Horn continuity for
landscape analysis by applying it to the quantum orbital minimization
method (qOMM). The qOMM corresponds to the optimization problem
\begin{equation*}
    \min_{X \in \ob(n,p)} E_0(X) = \tr\left( (2I - X^* X ) X^* A X \right),
\end{equation*}
where $A$ is a negative definite Hermitian matrix. To verify that a
stationary point $X$ is not a local minimizer, one must construct a
descent path $X(\veps)$ on the oblique manifold $\ob(n,p)$ in its
neighborhood. The strong Schur-Horn continuity theorem guarantees the
existence of such a path.

Denoting the Lagrange multipliers by a diagonal matrix $D = \diag(d_1,
\dots, d_p)$, the Lagrangian function for the qOMM problem is
\begin{equation*}
    L(X, D) = \tr \left( (2I - X^*X) X^*AX \right)
    + \tr(D^\top (X^*X - I)).
\end{equation*}
The corresponding first-order optimality condition are
\begin{subnumcases}{ \label{eq:Lagrange_qomm_zero_grad}} 
    2AX - AXX^*X - XX^*AX + XD = 0, \label{eq:Lagrange_qomm_zero_grad_X} \\
    X \in \ob(n,p). \label{eq:Lagrange_qomm_zero_grad_D}
\end{subnumcases}
From the first-order condition we could figure out all the stationary
points of qOMM. For illustrative purposes, we restrict our attention to
a particular class of stationary points. Let $Q_p \in \bbC^{n\times p}$
be the matrix whose columns are the eigenvectors of $A$ corresponding to
eigenvalues $\Lambda_p=\diag(\lambda_1, \cdots, \lambda_p)$. We consider
a rank-deficient stationary point $X \in \ob(n,p)$ with rank $p-1$ and
the corresponding Lagrange multiplier $D$ defined as follows. Let 
\begin{equation*}
    X = Q_p \Sigma V^*, \quad D = d I, \quad 
    d = \frac{2}{\sum_{i=1}^{p-1} \lambda_i^{-1}} < 0,
\end{equation*}
where
\begin{equation*}
    \Sigma=\diag(\sigma_1, \cdots, \sigma_{p-1}, \sigma_p), \quad 
    \sigma_i = \sqrt{1 + \frac{d}{2 \lambda_i}}, \quad i=1, \ldots, p-1, \quad \sigma_p = 0,
\end{equation*}
and $V$ is a unitary matrix such that $X \in \ob(n,p)$. The existence of
$V$ is guaranteed by the Schur-Horn theorem since $s:=\diag(\Sigma^2)$
is strictly majorized by $\mathbf{1}$, which follows from $\sigma_i > 1$
for $i=1, \ldots, p-1$. One could verify that $(X, D)$ satisfies the
first-order condition \eqref{eq:Lagrange_qomm_zero_grad} by direct
substitution.

We now construct a descent path $X(\veps) \in \ob(n,p)$ in the
neighborhood of $X$ such that for sufficiently small $\veps > 0$,
$E_0(X(\veps)) < E_0(X)$. Note that the objective function at the
stationary point $X=Q_p \Sigma V^*$ takes the value
\begin{equation*}
    E_0(X) = \tr\left( (2I - X^* X ) X^* A X \right) 
    = \sum_{i=1}^p \left( 2 - \sigma_i^2 \right) \sigma_i^2 \lambda_i.
\end{equation*}
Thus, we construct a perturbed point $\tX = Q_p \tSigma \tV^*$ such that
\begin{equation*}
    \tSigma^2 = (1 - \veps)\Sigma^2 + \veps I,
\end{equation*}
where $\veps > 0$ is sufficiently small and $\tV$ is a unitary matrix to
be determined. Let $H := X^*X = V \Sigma^2 V^*$ and $s :=
\diag(\Sigma^2)$. Since $s$ is strictly majorized by $\mathbf{1}$,
\cref{prop:sSH_major_hermit} implies that $H$ is strongly Schur-Horn
continuous. Moreover, since $\tr(\tSigma^2)=\tr(\Sigma^2)$ and
$\fnorm{\tSigma^2 - \Sigma^2} = O(\veps)$, there exists
\begin{equation*}
    \tH = G_2 V G_1 \tSigma^2 G_1^* V^* G_2^*
\end{equation*}
such that $\diag(\tH)=\diag(H)=\mathbf{1}$, $G_i$ is unitary, and
$\fnorm{G_i - I} = O(\veps^{1/2})$ for $i = 1, 2$. Let $\tV = G_2 V
G_1$. Then the distance between $V$ and $\tV$ is bounded by
\begin{equation*}
    \fnorm{\tV - V} = \fnorm{G_2 V G_1 - G_2 V + G_2 V - V} \leq
    \fnorm{G_1 - I} + \fnorm{G_2 - I} = O(\veps^{1/2}),
\end{equation*}
where we use the unitary invariance of Frobenius norm. Hence, we know
that $\tX \in \ob(n,p)$ since $\diag(\tX^* \tX) = \diag(\tH) =
\mathbf{1}$. Moreover, the distance between $X$ and $\tX$ could be
bounded as,
\begin{equation*}
    \fnorm{\tX - X} = \fnorm{Q_p \tSigma \tV^* - Q_p \Sigma V^*}
    \leq \fnorm{\tSigma - \Sigma}
    + \fnorm{\Sigma} \fnorm{\tV^* - V^*} = O(\veps^{1/2}).
\end{equation*}
Finally, for $\veps > 0$ sufficiently small, we have
\begin{equation*}
    E_0(\tX) - E_0(X) = \sum_{i=1}^{p-1} \lambda_{i}(2 - \veps) \veps(\sigma_{i}^2 -1)^2
    + \lambda_{p}(2 - \veps) \veps < 0,
\end{equation*}
since $\lambda_i < 0$ for $i=1, \ldots, p$. Thus, for sufficiently small
$\veps > 0$, there exists $\tX \in \ob(n,p)$ such that $\fnorm{\tX - X}
= O(\veps^{1/2})$ and $E_0(\tX) < E_0(X)$. Such rank-deficient
stationary points could not be local minimizers.

\bibliographystyle{siam}
\bibliography{continuitySH}

\end{document}